\documentclass[a4paper,11pt]{article} 

\usepackage{amsmath, amscd, amsthm, amssymb, mathtools} 
\usepackage{hyperref}
\usepackage[applemac]{inputenc}
\usepackage[all]{xy} 
\usepackage[T1]{fontenc} 
\usepackage{textcomp} 
\usepackage[margin=3cm]{geometry}

\usepackage{authblk}

\usepackage{mathrsfs}

\usepackage{concrete}

\DeclareMathOperator{\im}{im}

\DeclareMathOperator{\Lie}{Lie}

\DeclareMathOperator{\Hom}{Hom}

\DeclareMathOperator{\tr}{Tr}

\DeclareMathOperator{\Diff}{Diff}

\DeclareMathOperator{\SO}{SO}

\DeclareMathOperator{\U}{U}

\DeclareMathOperator{\Rm}{Rm}
\DeclareMathOperator{\Ric}{Ric}

\DeclareMathOperator{\Vect}{Vect}

\newcommand{\R}{\mathbb R}
\newcommand{\C}{\mathbb C}

\newcommand{\N}{\mathbb N}

\newcommand{\diff}{\mathrm{d}}
\newcommand{\del}{\partial}

\newcommand{\Id}{\text{Id}}

\newcommand{\dvol}{\mathrm{dvol}}

\newcommand{\zs}{\prescript{0}{}\sigma}

\newcommand{\G}{\mathcal{G}}
\newcommand{\so}{\mathfrak{so}}

\renewcommand{\div}{\text{\rm div}}

\renewcommand{\H}{\mathbb H}

\renewcommand{\u}{\mathfrak{u}}
\renewcommand{\L}{\mathcal{L}}

\theoremstyle{plain}
	\newtheorem{theorem}{Theorem}
	\newtheorem{proposition}[theorem]{Proposition}
	\newtheorem{lemma}[theorem]{Lemma}
	\newtheorem{corollary}[theorem]{Corollary}

\theoremstyle{definition}
	\newtheorem{definition}[theorem]{Definition}
	\newtheorem{remark}[theorem]{Remark}
	\newtheorem{remarks}[theorem]{Remarks}

\theoremstyle{plain}
	\newtheorem*{theorem*}{Theorem}
	\newtheorem*{proposition*}{Proposition}
	\newtheorem*{lemma*}{Lemma}
	\newtheorem*{corollary*}{Corollary}
	\newtheorem*{conjecture*}{Conjecture}
	\newtheorem*{inductive-hypotheses}{Inductive Hypotheses}
	\newtheorem*{base-step}{Base Step}
	\newtheorem*{inductive-step}{Inductive Step}
	\newtheorem*{final-step}{Final Step}
\theoremstyle{definition}
	\newtheorem*{definition*}{Definition}
	\newtheorem*{remark*}{Remark}
	\newtheorem*{remarks*}{Remarks}
	
\makeatletter
\def\blfootnote{\xdef\@thefnmark{}\@footnotetext}
\makeatother

\numberwithin{equation}{section}
\numberwithin{theorem}{section}

\begin{document}

\title{Non-degeneracy of Poincare--Einstein four-manifolds satisfying a chiral curvature inequality}

\author{Joel Fine}%\\ \small{Université libre de Bruxelles, Belgium}}

\date{ }

\maketitle

\abstract{A Poincaré--Einstein metric $g$ is called \emph{non-degenerate} if there are no non-zero infinitesimal Einstein deformations of $g$, in Bianchi gauge, that lie in $L^2$. We prove that a 4-dimensional Poincaré--Einstein metric is non-degenerate if it satisfies a certain chiral curvature inequality. Write $\Rm_+$ for the part of the curvature operator of $g$ which acts on self-dual 2-forms. We prove that if $\Rm_+$ is negative definite then $g$ is non-degenerate. This is a chiral generalisation of a result due to Biquard and Lee, that a Poincaré--Einstein metric of negative sectional curvature is non-degenerate.}

%\tableofcontents

\section{Introduction}\label{introduction}

Let $\overline{M}$ be a compact $(n+1)$-manifold with boundary $X$ and interior $M$ and let $\rho \colon \overline{M} \to [0,\infty)$ be a boundary defining function (so that $\rho^{-1}(0) = X$ and $\rho$ vanishes transversely there). A Riemannian metric $g$ on $M$ is called \emph{conformally compact} if $\rho^2 g$ extends to a smooth metric on $\overline{M}$. If, in addition, $\Ric(g) = -ng$  then $g$ is called \emph{Poincaré--Einstein}. The prototype is the Poincaré model of hyperbolic space, on the unit ball $\overline{B}^{n+1} = \{ x \in \R^{n+1} : |x| \leq 1\}$ with metric $g = \rho^{-2}\diff x^2$ where $\rho = \frac{1}{2}(1-|x|^2)$.

Poincaré--Einstein metrics were introduced in the mathematics literature by Fefferman and Graham \cite{Fefferman-Graham}, where they have proved invaluable in conformal geometry. It was subsequently realised, following the work of Maldacena \cite{Maldacena} on holography, that  they are also significant in theoretical physics. Since these seminal works, Poinacré--Einstein metrics have been the subject of a great deal of research, in both the  mathematical and physical communities. 

A Poincaré--Einstein metric $g$ is called \emph{non-degenerate} if, once one gauge fixes for the action of diffeomorphisms, there are no non-zero infinitesimal Einstein deformations of $g$ which are in $L^2$. (The precise definition is given below; see Definition~\ref{definition-non-degeneracy}.) The main result of this article is that when a four-dimensional Poincaré--Einstein metric satisfies a certain pointwise curvature inequality, then $g$ is automatically non-degenerate. We will give the inequality shortly, but first we explain the geometric importance of non-degeneracy.

A central question in the study of Poincaré--Einstein metrics is the so-called Dirichlet problem. The metric $\rho^2g|_X$ on $X$ depends on the choice of boundary defining function $\rho$ but the conformal class of $\rho^2g|_X$ is uniquely determined by~$g$; we denote it by $c(g)$ and call it the \emph{conformal infinity of~$g$}. For example, the conformal infinity of hyperbolic space $\H^{n+1}$ is the standard conformal structure on $S^n$. Given a conformal structure $\gamma$ on $X$, the Dirichlet problem is to find a Poincaré--Einstein $g$ on $M$ with $c(g) = \gamma$. 

One of the central results in this direction, due in various forms to Graham and Lee~\cite{Graham-Lee}, Lee~\cite{Lee} and Biquard~\cite{Biquard}, is that if $g$ is a non-degenerate Poincaré--Einstein metric, then the \emph{local} Dirichlet problem always has a unique solution: if $\gamma$ is a conformal structure on $X$ which is sufficiently close to $c(g)$ then there is a Poincaré--Einsten metric $\tilde{g}$ with $c(\tilde{g})=\gamma$; moreover, $\tilde{g}$ is the unique solution modulo diffeormorphisms which is close to $g$. Graham and Lee \cite{Graham-Lee} showed that the hyperbolic metric on $\H^{n+1}$ is non-degenerate. This was followed by Biquard \cite{Biquard} and Lee \cite{Lee} who showed that any Poincaré--Einstein metric of non-positive sectional curvature is automatically non-degenerate.
 
As we have mentioned, the main result of this article is to give a different curvature inequality, particular to $\dim M=4$, which implies non-degeneracy. To state the inequality, let $(M,g)$ be an oriented Poincaré--Einstein 4-manifold. The bundle of 2-forms splits $\Lambda^2 = \Lambda^+\oplus \Lambda^-$ into $\pm1$-eigenbundles for the Hodge star. The curvature operator $\Rm \colon \Lambda^2 \to \Lambda^2$ splits $\Rm = \Rm_+ + \Rm_-$ where $\Rm_+$ and $\Rm_-$ are self-adjoint endomorphism of $\Lambda^+$ and $\Lambda^-$ respectively. (The fact that there is no part of $\Rm$ which swaps self-dual and anti-self-dual forms is equivalent to $g$ being Einstein.) The main result of this article is the following:

\begin{theorem}\label{main-result}
If $(M,g)$ is an oriented Poincaré--Einstein 4-manifold and either $\Rm_+ < 0$ or $\Rm_- < 0$ then $g$ is non-degenerate.  
\end{theorem}

We make three quick remarks about this result.

\begin{remarks}
~\begin{enumerate}
\item
One can check that for \emph{any} Poincaré--Einstein 4-manifold, $\Rm_\pm \to - \Id$ as $\rho \to 0$ so at least near infinity the curvature condition is automatically satisfied. 
\item
A previous article \cite{FKS} considered exactly this inequality for a \emph{compact} Einstein 4-manifold $(M,g)$, with the analogous conclusion: the only infinitesimal Einstein deformations of $g$ are given by Lie derivatives. The proof of  Theorem~\ref{main-result}  follows the ideas of \cite{FKS} closely, but new arguments are needed in several places to deal with the fact that $M$ is not compact. Moreover, the relevant equations become \emph{degenerate} at the boundary, in the sense that the symbol vanishes there. This means that the classical elliptic theory is no longer available.  Instead we will use the $0$-calculus of Mazzeo and Melrose \cite{Mazzeo,Mazzeo-Melrose}.
\item
To compare this inequality to one involving sectional curvatures, let $\lambda_{\pm}$ be the largest eigenvalues of $\Rm_{\pm}$. Then the hypothesis in Theorem~\ref{main-result} requires one of $\lambda_+$ or $\lambda_-$ to be negative whilst non-positive sectional curvatures means $\lambda_+ + \lambda_- \leq 0$. The difference between the inequalities is also evident in examples. In  \cite{Calderbank-Singer}, Calderbank and Singer construct \emph{anti-self-dual} Poincaré--Einstein metrics on many different 4-manifolds which all have non-trivial~$\pi_2$. Anti-self-dual is equivalent to $\Rm_+ = - \Id$ which certainly satisfies our inequality. Meanwhile the fact that $\pi_2 \neq 0$ means that these manifolds carry no complete metrics of non-positive curvature whatsoever. 
\end{enumerate}
\end{remarks}

The rest of the paper has two sections. In \S\ref{background} we describe the technical background needed: \S\ref{0-calculus} reviews the $0$-calculus, designed to treat uniformly degenerate PDE of the kind which appear here; \S\ref{Bianchi-gauge-review} recalls the gauge-fixed Einstein equations and the definition of non-degeneracy; \S\ref{def-conn} describes definite connections, which are a way of parametrising Einstein metrics for which one of $\Rm_+$ or $\Rm_-$ is a definite endomorphism. In \S\ref{the-proof} we give the proof of Theorem~\ref{main-result}. %There are two parts: the main one establishes the analogous result for deformations of torsion-free definite connections; then we deduce the result for the metrics themselves. 

\subsubsection*{Acknowledgements} 

I would like to thank Rafe Mazzeo and Michael Singer for several helpful conversations on this topic. This research were supported by the ERC consolidator grant ``SymplecticEinstein'' 646649 and the Excellence of Science grant 4000725.

\section{Background material}\label{background}

\subsection{Preliminaries on 0-elliptic differential operators}
\label{0-calculus}

In various places our proofs will use geometric elliptic operators on Poincaré--Einstien manifolds. The relevant analytic framework is the so-called 0-calculus, developed by Mazzeo and Melrose \cite{Mazzeo,Mazzeo-Melrose}. We give a rapid review of the parts we need here.

Recall that $\overline{M}$ denotes a compact $(n+1)$-dimensional manifold with boundary $X$ and interior $M$. Fix a collar neighbourhood of the boundary $U \cong X \times [0,1]$. Write $\rho$ for the boundary defining function pulled back from $[0,1]$ and $x_1, \ldots ,x_n$ for coordinates on $X$. Together, $\rho, x_i$ are local coordinates on $\overline{M}$. A 0-differential operator of order $d$ acting on functions is a differential operator which locally has the form
\[
D = \sum_{i+|\beta|\leq d} c_{i,\beta}(\rho, x)(\rho \del_\rho)^i(\rho\del_x)^{\beta}
\]
where the coefficient functions are smooth up to the boundary. More invariantly, write $\Vect_0$ for the set of all tangent vector fields on $\overline{M}$ which are tangent to the boundary. The class of $0$-differential operators acting on functions is the sub-algebra of differential operators generated by $\Vect_0$ over $C^\infty(\overline{M})$. One can similarly talk about $0$-differential operators acting on vector-valued functions or sections of bundles, in which the coefficients $c_{i,\beta}$ of $D$ (written in local trivialisations) take values in matrices.

We will consider $0$-differential operators acting on weighted Banach spaces. Fix a conformally compact metric $g$ on $M$. Note that $\Vect_0$ contains precisely those vector fields of bounded norm with respect to $g$. We write $C^{k,\alpha}$ and $W^{p,k}$ for the Hölder and Sobolev spaces defined by $g$. Given a weight $\nu \in \R$ we write $C^{k,\alpha}_\nu$ for the set of functions of the form $\rho^\nu f$ where $f \in C^{k,\alpha}$ and similarly for $W^{p,k}_\nu$. A $0$-differential operator $D$ of order $d$ preserves the weighted spaces, giving maps $D \colon C^{k+d,\alpha}_{\nu} \to C^{k,\alpha}_\nu$ and $D\colon W^{p,k+d}_\nu \to W^{p,k}_{\nu}$. We now describe conditions under which these maps are Fredholm.

The \emph{$0$-symbol} of $D$ in the direction $(\xi, \eta_1, \ldots , \eta_n)$ is defined to be
\[
\zs(D)(\rho,x)(\xi,\eta)
  =
    \sum_{i + |\beta|=d} c_{i,\beta}(\rho,x) \xi^i\eta^\beta
\]
%To define this invariantly, one needs the notion of $0$-tangent bundle and $0$-cotangent bundle. The $0$-tangent bundle $\zTM \to \overline{M}$ is a rank-$(n+1)$ vector bundle whose space of sections is isomorphic to $\Vect_0$ as a $C^\infty(\overline{M})$-module. The existence of such a bundle is guaranteed by the Serre--Swann Theorem. More concretely, we can declare $\zTM$ to be trivial over any local coordinate patch $(\rho, x_1, \ldots, x_n)$ and define the transition functions by the transformation rules obeyed by the $n+1$ vector fields $\rho \del_\rho, \rho \del_{x_i}$. However one makes the definition, there is a homomorphism of bundles $\zTM \to T\overline{M}$ which is an isomorphism on the interior and zero on the boundary. The $0$-cotangent bundle is the dual $\zCTM$ of $\zTM$. It is locally spanned by the covectors $\rho^{-1}\diff \rho, \rho^{-1}\diff x_i$. With this in hand, one can check that the $0$-symbol of a $0$-differential operator of order $d$ makes invariant sense as a section of $S^d(\zCTM)$. In the case when $D$ sends sections of $V$ to sections of $W$, the $0$-symbol of~$D$ takes values in $\Hom(V,W)$.
%
A $0$-differential operator $D$ is called \emph{$0$-elliptic} if $\zs(D)(\xi,\eta)$ is an isomorphism whenever $(\xi, \eta) \neq 0$. This alone is not sufficient, however, to guarantee good mapping properties. The reason is that the \emph{traditional} symbol vanishes at $\rho= 0$ and so the usual methods for constructing local inverses to $D$ fail at the boundary. To extend the local analysis of $D$ to $X$ one must also consider the \emph{normal operator} of $D$. There is a normal operator for each boundary point; at $(0,x_1,\ldots, x_n)$ it is the differential operator $N(D)$ defined on the half-space $\{(r,s_1, \ldots s_n): r \geq 0\}$ by
\[
N(D)(0,x) 
  = 
    \sum_{i + |\beta|\leq d} 
      c_{i,\beta}(0,x)
      (r \del_r)^i(r \del_s)^\beta
\]
More invariantly, $N(D)(0,x)$ is defined over the inward pointing half-space in $T_{(0,x)}\overline{M}$. Choose coordinates $(\rho,x)$ centred at the boundary point of interest, inducing coordinates $(r,s)$ on $T_{(0,0)}\overline{M}$. The family of homotheties $(\rho,x) = \epsilon(r,s)$ for decreasing $\epsilon$ identifies larger and larger regions in the half-space with a smaller and smaller neighbourhood of $(0,0) \in \overline{M}$ and we can use this to make $D$ act on functions with compact but increasingly large support in $T_{(0,0)}\overline{M}$. In an appropriate limit as $\epsilon \to 0$ we obtain the normal operator. When $D$ sends sections of $V$ to sections of $W$, we must also trivialise the bundles near $(0,0)$ to make sense of this procedure; the limiting operator $N(D)$ then sends sections of the trivial bundle with fibre $V_{(0,0)}$ to sections of the trivial bundle with fibre $W_{(0,0)}$. 

\begin{remark}\label{geometric-normal-operators}
Suppose for a moment that $g$ is a conformally compact metric on $M$, and $D$ is an associated geometric $0$-differential operator (i.e.\ determined entirely by the Levi--Civita connection and metric and algebraic contractions). Then the normal operator $N(D)$ is simply the analogous operator defined on hyperbolic space. In particular it does not depend on the boundary point. The reason is that if we equip the half-space with the hyperbolic metric $r^{-2}(\diff r^2 + \diff s^2)$, then the homotheties $(\rho, x) =\epsilon (r,s)$ are asymptotically isometric in the limit $\epsilon \to 0$ (at least up to an overall scale).  
\end{remark}  

Finally we will also need the indicial polynomial of $D$. This is again defined at each boundary point; at $(0,x_1, \ldots, x_n)$ it is the polynomial in $\lambda \in \C$ given by
\[
I(D)(0,x)(\lambda)
  =
    \sum_{i+ |\beta| \leq d}
      c_{i, \beta}(0,x) \lambda^i  
\]
Taken invariantly, when $D$ sends sections of $V$ to sections of $W$, this is a polynomial with values in $\Hom(V_{(0,x)},W_{(0,x)})$. The \emph{indicial roots of $D$} are those $\lambda$ for which $I(D)$ has kernel. Again, these may in general depend on~$x$ but for a geometric operator they do not. 

We write $D^*$ for the formal adjoint of $D$ acting on $L^2$. The $0$-symbols, normal operators and indicial polynomials are related by $\zs(D^*)= \zs(D)^*$, $N(D^*)=N(D)^*$ and $I(D^*)(\lambda)=I(D)(n-\lambda)$. In particular the indicial roots of $D^*$ are $n - \lambda_i$ where the $\lambda_i$ are the indicial roots of~$D$.

%The indicial roots of $D$ have the following significance: when $k$ is in the kernel of $I(D)(\lambda)$, the locally defined section $\rho^\lambda k$ has the property that $D(\rho^\lambda k) = O(\rho^{\lambda+1})$. From this one sees that to have good mapping properties between weighted spaces it will be important for the weight to avoid the indicial roots. 

With this background in place, we can now state the main result on $0$-elliptic operators which we will use. For $\nu \in \R$ we consider $D$ and $D^*$ between the corresponding weighted spaces (note the different weights used for the adjoint operator):
\begin{alignat}{3}
& D\colon C^{k+d,\alpha}_{\nu}(M) \to C^{k,\alpha}_{\nu}(M)
\qquad
&& \text{and}
\qquad
&& D\colon W^{p,k+d}_{\nu-n/2}(M) \to W^{p,k}_{\nu-n/2}(M)
\label{D-maps}\\
& D^*\colon C^{k+d,\alpha}_{n-\nu}(M) \to C^{k,\alpha}_{n-\nu}(M)
\qquad
&& \text{and}
\qquad
&& D^*\colon W^{p,k+d}_{n/2-\nu}(M) \to W^{p,k}_{n/2-\nu}(M)
\label{Dstar-maps}
\end{alignat}

\begin{theorem}[\cite{Mazzeo}]\label{mazzeo-fredholm}
Let $D$ be a $0$-elliptic operator of order $d$. We assume that
\begin{itemize}
\item The indicial roots of $D$ are constant.
\item There is no indicial root of $D$ whose real part lies in the interval $(\lambda_1,\lambda_2)$.
\item For some $\nu \in (\lambda_1, \lambda_2)$ the normal operator is invertible as a map 
\[
N(D) \colon W^{2,d}_{\nu -n/2}(\H^{n+1}) \to L^2_{\nu-n/2}(\H^{n+1}). 
\]
\end{itemize}
Then 
\begin{enumerate}
\item
For all $\nu \in (\lambda_1, \lambda_2)$, any $p>1$, any $k \in \N$ and any $\alpha \in (0,1)$, the maps in \eqref{D-maps} and \eqref{Dstar-maps} are Fredholm, with index independent of $\nu$. 
\item
The kernels of the maps in \eqref{D-maps} and \eqref{Dstar-maps} are all equal,  independent of $\nu \in (\lambda_1, \lambda_2)$ and whether we are considering Hölder or Sobolev spaces. 
\item
The kernel of the maps in \eqref{Dstar-maps} project isomorphically onto the cokernel of the maps~\eqref{D-maps}, and vice versa.
\end{enumerate}
\end{theorem}

\subsection{The Einstein equation in Bianchi gauge}
\label{Bianchi-gauge-review}

The Einstein equation $\Ric(g) + ng=0$ is invariant under the action of diffeomorphisms and so is not elliptic. One of the standard ways to deal with this issue is to use Bianchi gauge, which we briefly recall here. The \emph{Bianchi operator} of the metric $g$ is the map $B_g \colon \Gamma(S^2T^*M) \to \Gamma(T^*M)$  defined by
\begin{equation}\label{Bianchi-operator}
B_g(h) = \div_g(h) + \frac{1}{2} \diff \left(\tr_g h\right).
\end{equation}
Here, $(\div_gh)_i = -\nabla^j h_{ij}$, where $\nabla$ is the Levi-Civita connection of $g$. The Bianchi operator can be combined  with the Einstein condition as follows. Given a conformally compact metric $\tilde{g}$, define $F_g(\tilde{g})$ by 
\begin{equation}
F_g(\tilde{g}) = \Ric(\tilde{g}) + n\tilde{g} +\div_{\tilde{g}}^*\left(B_g(\tilde{g})\right).
\label{gauge-fixed-Einstein}
\end{equation}
Here $(\div^*_g\alpha)_i = \nabla_{(i}\alpha_{j)}$.

We write $L_g$ for the linearisation of $F_g$ at $g$. This has the form
\[
L_g = D_g + \div_g^*\circ B_g
\]
where $D_g$ is the linearisation of $\tilde{g} \mapsto \Ric(\tilde{g})+n\tilde{g}$ at $g$. The point is that $D_g$ is not elliptic, since its kernel contains all Lie derivatives: $D_g(L_vg)=0$; however the $\div_g^* \circ B_g$ term precisely compensates for this, leading to an \emph{elliptic} linearisation.

The following is due, in various forms, to Biquard, Graham and Lee:

\begin{proposition}[\cite{Biquard,Graham-Lee,Lee}]
Let $(M,g)$ be a Poincaré--Einstein $(n+1)$-manifold. The operator $L_g$ is 0-elliptic, and formally self-adjoint, with indicial roots $-1,0,n$ and $n+1$. It has invertible normal operator on $L^2$ and so is Fredholm on weighted spaces with weight in $(0,n)$. 
\end{proposition}

Since $L_g$ is formally self-adjoint it has vanishing index. This means that $L_g$ is an \emph{isomorphism} between these weighted spaces precisely when it has no kernel in $L^2$. This motivates the definition of non-degeneracy. 

\begin{definition}\label{definition-non-degeneracy}
A Poinacré--Einstein metric $g$ is called \emph{non-degenerate} if there is no non-zero solution to $L_g(h)=0$ with $h \in L^2$. Equivalently, $g$ is non-degenerate if $L_g \colon W^{2,2} \to L^2$ is an isomorphism. 
\end{definition}

Solutions to the gauge fixed equations~\eqref{gauge-fixed-Einstein} are indeed Einstein metrics. This standard but important fact (and, in particular, its linearisation) will be important in what follows. We state it here, with a sketch of the proof.
\begin{proposition}\label{gauge-fixing-does-job}
Let $(M,g)$ be a Poinacré--Einstein manifold. 
	\begin{itemize}
	\item
	If $\tilde{g}$ is another conformally compact metric, close to $g$ in $C^2$, with $F_g(\tilde{g}) =0$ then $\Ric(\tilde{g}) + n \tilde{g}=0 = \div^*_{\tilde{g}}(B_g(\tilde{g}))$.
	\item
	If $h\in L^{2} \cap \ker L_g$, then $D_g(h)=0 = \div_g^* B_g(h)$. 
	\end{itemize}
\end{proposition}
\begin{proof}[Sketch of proof]
Suppose that $F_g(\tilde{g})=0$. By the Bianchi identity, $B_{\tilde{g}}(\Ric(\tilde{g})) = 0$; meanwhile one checks directly that $B_{\tilde{g}}(\tilde{g}) =0$ and so we get from $B_{\tilde{g}}(F_g(\tilde{g}))=0$ that
\begin{equation}
B_{\tilde{g}}(\div^*_{\tilde{g}}(B_g(\tilde{g}))) = 0
\label{B-div-B}
\end{equation}
A calculation shows that 
\begin{equation}
2 B_{\tilde{g}} \circ \div^*_{\tilde{g}} = \nabla_{\tilde{g}}^*\nabla_{\tilde{g}} - \Ric(\tilde{g}) 
\label{B-div-Bochner}
\end{equation}
Now suppose that $\Ric(\tilde{g}) \leq - c \tilde{g}$ for $c>0$ (which is true if $\tilde{g}$ is close enough to $g$ in $C^2$) and let $\alpha \in W^{2,2}(T^*M)$ with $B_g(\div^*_g \alpha)=0$. Then integrating the above equation against $\alpha$ we obtain that 
\[
0 \geq \int \left( |\nabla_{\tilde{g}} \alpha|^2  + c |\alpha|^2 \right)\dvol_{\tilde{g}}
\]
Hence $\alpha =0$.  This, together with equation~\eqref{B-div-B}, implies that $B_g(\tilde{g}) = 0$ and so $\Ric(\tilde{g}) + n \tilde{g}=0$ also.

For the second part we differentiate this argument.  Firstly note that since $h \in \ker L^2 \cap L_g$ then in fact $h \in W^{p,k}$ for any $p>1, k \geq 0$. Secondly, for any metric $\tilde{g}$, we have  $B_{\tilde{g}}(\Ric(\tilde{g})+n\tilde{g})=0$. Differentiating this with respect to the metric and evaluating at $g$, and using the fact that $\Ric(g)+ng=0$, we obtain $B_g \circ D_g = 0$. So applying $B_g$ to the equation $0= D_g(h) + \div_g^*B_g(h)$ we see that $B_g\div^*_g(B_g(h))=0$. Now $B_g(h) \in W^{2,2}$ (since $h \in W^{2,3}$) and so equation~\eqref{B-div-Bochner} again shows that $B_g(h)=0$ and hence $D_g(h)=0$ also. 
\end{proof}

\subsection{Definite connections}\label{def-conn}

The proof of Theorem~\ref{main-result} relies on an alternative approach to 4-dimensional Einstein metrics which is due independently to the author \cite{Fine} and Krasnov \cite{Krasnov}. We describe it briefly here. A comprehensive account aimed at a mathematical audience can be found in \cite{Fine-Krasnov-Panov}; a briefer treatment, which follows more closely the viewpoint used here, is given in \cite{FKS}. These references also explain how this approach gives a variational formulation of Einstein metrics, a side of the story which we ignore here. 

The main idea is to parametrise Riemannian metrics by (an open subset of) $\SO(3)$-conn\-ections, loosely analogous to the relationship between the electormagnetic field and electromagnetic potential (thought of as a $\U(1)$-connection). Let $M$ be a 4-manifold, $E \to M$ an oriented $\SO(3)$-bundle and $A$ an $\SO(3)$-connection in $E$. The metric and fibrewise volume-form give a series of isomorphisms 
\[
E^* \cong E \cong \Lambda^2E^* \cong \so(E) \cong \so(E)^*.
\] We will use these to treat the curvature $F_A \in \Omega^2(M,\so(E))$ sometimes as a section of $\Lambda^2\otimes E$ and sometimes as a homomorphism $E \to \Lambda^2$. We will often make these kinds of identifications without comment. 

The following definition appears explicitly in \cite{Fine-Panov}, but it is also a special case of the ``fat connections'' introduced by Weinstein \cite{Weinstein}.
\begin{definition}\label{definite-connection}
~
\begin{itemize}
\item
An $\SO(3)$-connection $A$ in $E$ is called \emph{definite} if for any point $x \in M$ and for any pair of independent tangent vectors $u,v \in T_xM$, the infinitesimal rotation $F_A(u,v) \in \so(E_x)$ is non-zero. 
\item
Equivalently, $A$ is definite if the image of $F_A \colon E \to \Lambda^2$ is a maximal definite subspace for the wedge-product on $\Lambda^2$. In more detail: the wedge product defines a symmetric bilinear form $\Lambda^2\times \Lambda^2 \to \Lambda^4$; dividing by a nowhere vanishing 4-form $\mu$ makes this a non-degenerate inner-product of signature $(3,3)$; we ask that $\im F_A$ has rank~3, and that the wedge-product is definite here.
\item
Equivalently, if $e_i$ is a local orthonormal frame for $E$ in which the curvature becomes $F_A = \sum F_i \otimes e_i$ for a triple of 2-forms $F_i$, and if $\mu$ is an auxiliary choice of nowhere vanishing 4-form, then the following symmetric matrix-valued function is definite (i.e.\ all eigenvalues have the same sign):
\[
Q(\mu)_{ij} = \frac{F_i \wedge F_j}{2\mu}.
\]
More invariantly, $Q(\mu)$ is a definite self-adjoint endomorphism of $E$.

\end{itemize}
\end{definition}

The key feature of a definite connection $A$, for our purposes at least, is that \emph{it determines in a canonical way a Riemannian metric $g_A$ on $M$}. The definition of $g_A$ goes in three steps. 
\begin{enumerate}
\item
$A$ determines an orientation on $M$: a volume form $\mu$ is positively oriented if $Q(\mu)$ is \emph{positive} definite. 
\item
There is a unique conformal structure on $M$ for which $A$ is a self-dual instanton, namely that for which the image of $F_A \colon E \to \Lambda^2$ is the  bundle of self-dual 2-forms. We write $\Lambda^+_A$ for the image of $F_A$. In a local frame $e_i$ of $E$ in which $F_A = \sum F_i \otimes e_i$, we have  $\Lambda^+_A = \left\langle F_i \right\rangle$. Notice that the isomorphism $E \to \Lambda^+_A$ means that the the topological type of the bundle $E$ is a priori fixed. 
\item
$A$ determines a volume form $\mu_A$ as follows. Given an abritrary volume form $\mu$, the self-adjoint endomrophism $Q(\mu)$ and $\mu$ scale inversely. We single out a particular volume form, denoted $\mu_A$, by requiring that 
\[
\tr\sqrt{ Q(\mu_A)} = 3.
\] 
(Here $\sqrt{Q}$ denotes the positive square root of $Q$.) We will explain the reason for this choice in Remark~\ref{why-tr-sqrt-3} below. We write $g_A$ for the unique Riemannian metric on $M$ with volume form $\mu_A$ and for which $F_A$ is self-dual.
\end{enumerate}

\begin{definition}
The \emph{gauge group} $\G$ is the group of all diffeomorphisms $E \to E$ which send fibres to fibres by linear isometries. Equivalently, if $P$ denotes the prinicpal $\SO(3)$-bundle of frames in $E$, then $\G$ is the group of diffeomorphisms of $P$ commuting with the principal $\SO(3)$-action. The group $\G$ fits into a short exact sequence
\begin{equation}
1 \to \G_V \to \G \to \Diff(M) \to 1
\label{G-extension}
\end{equation}
where $\G_V$ is the subgroup of ``vertical'' gauge transformations, i.e.\ those that cover the identity on $M$. The group $\G$ acts by pull-back on the space of all $\SO(3)$-connections in $E$; this action preserves the open set of definite connections. 
\end{definition}

We will need the formula for the infinitesimal action $R_A \colon \Lie(\G) \to \Omega^1(M,E)$ of $\G$ at a connection $A$. To describe it, we note first that the Lie algebra of $\G$ is a special subalgebra of the vector fields on $E$. Given a connection $A$ in $E \to M$, any vector field on $E$ splits into horizontal and vertical parts. If $\xi \in \Lie(\G)$ it has a special form: $\xi = u + v_A$ where $v_A$ is the $A$-horizontal lift of some $v \in \Gamma (M,TM)$ and the vertical part is $u \in \Gamma(M,\so(E))$. 

\begin{lemma}\label{infinitesimal-gauge-action}
Suppose $\xi \in \Lie(\G)$ has the splitting $\xi = u + v_A$ with respect to a connection $A$. The infinitesimal action of $\xi$ at a $A$ is given by $R_A(\xi) = \diff_A u + \iota_v F_A$.
\end{lemma}
\begin{proof}
This is easiest to check in the principal bundle formalism. There $\xi$ is a vector field on $P$ and the connection $A$ is an equivariant $\so(3)$-valued 1-form. Meanwhile, $u =A(\xi)$ is an $\SO(3)$-equivariant map $P \to \so(3)$. The infinitesimal action is the Lie derivative: 
\[
R_A(\xi) 
	= 
		L_\xi(A) 
	= 
		\diff (A(\xi)) + \iota_\xi (\diff A)
	=
		\diff u + \iota_\xi(F_A) - \frac{1}{2}\iota_\xi [A\wedge A]
	=
		\diff_A u + \iota_\xi(F_A)
	\qedhere
\] 	
\end{proof}

We will use definite connections to parametrise a subset of the Riemannian metrics on $M$, via the map $A \mapsto g_A$. %As remarked at the end of Definition~\ref{definite-connection}, this parametrisation is equivariant with respect to the action of $\G$. It is important to note, however, that we certainly do not get all metrics this way, not even locally. One way to see this is by a count of functional dimension. $\SO(3)$-connections are locally given by twelve functions and definite connections are an open set in the space of all connections, so the space of definite connections has functional dimension 12. Meanwhile, $\G$ has functional dimension 7 (in terms of the short exact sequence \eqref{G-extension} three functions account for $\G_V$, i.e.\ the rotations of the fibres, and four functions give the diffeomorphisms of the base). This means that the space of definite connections modulo gauge has functional dimension~5. On the other hand, Riemannian metrics are locally given by ten functions, and diffeomorphisms by four and so the space of metrics modulo diffeomorphisms has functional dimension~6. It can be shown that the map $A \mapsto g_A$ is a kind of ``slice'' for conformal rescalings: the map sending a definite connection $A$ to the conformal class of $g_A$ is Fredholm (so locally  a genuine parametrisation up to finite dimensional discrepancy). Since we will not use this fact we omit the justification. 
We now explain how to identify those $A$ for which $g_A$ is Einstein. We begin by defining the \emph{torsion} of a definite connection. To ease the notation we write $Q_A = Q(\mu_A)$. This self-adjoint endomorphism of $E$ relates the metric on $E$ to the pull-back of the metric on $\Lambda^+_A$ via the isomorphism $F_A \colon E \to \Lambda^+_A$. Let 
\begin{equation}\label{Sigma-definition}
\Sigma_A =\pm F_A \circ Q_A^{-1/2}.
\end{equation} 
We will explain shortly how to fix the choice of sign in this equation, but we momentarily leave it ambiguous. With respect to a local frame in which $Q^{-1/2}_A = \sum P_{ij}e_i\otimes e_j$, we have $\Sigma_A = \sum \Sigma_i \otimes e_i$ where $\Sigma_i = \pm \sum_j P_{ij}F_j$. The point is that whilst the $F_i$ are merely a basis for $\Lambda^+_A$, the $\Sigma_i$ are a $g_A$-\emph{orthonormal} basis. In other words, $\Sigma_A \colon E \to \Lambda^+_A$ is an isometry.

We now address the sign in the definition of $\Sigma_A$. Consider the (locally defined) triple of homomorphisms $J_i \colon T^*M \to T^*M$ given by
\begin{equation}\label{J-definition}
J_i(\alpha) = *\sqrt{2} (\alpha \wedge \Sigma_i)
\end{equation}
Since the $\Sigma_i$ are unit length and self-dual, the $J_i$ are almost complex structures, i.e.~$J_i^2=-1$. Since the $\Sigma_i$ are orthogonal the $J_i$ satisfy the quaternion relations \emph{up to sign}:
\[
J_1J_2J_3 = \pm \Id.
\]
This sign does not depend on the choice of local oriented frame $e_i$. If one  goes carefully through the identifications, one sees that it does not depend on the orientation of $E$ either (because this appears in the isomorphism $\so(E) \cong E$). The sign depends only on the definite connection $A$ itself.  

\begin{definition}
The sign in the definition~\eqref{Sigma-definition} of $\Sigma_A$ is chosen so that $J_1J_2J_3 = -\Id$. When $+$ is used in \eqref{Sigma-definition} we call $A$ \emph{positive definite}, when $-$ is used we call $A$ \emph{negative definite}.
\end{definition}

Next we push $A$ forward to $\Lambda^+_A$ via $\Sigma_A$, to obtain a metric connection $\hat{A}$ in $\Lambda^+_A$. The torsion $\tau(\hat{A})$ of a connection $\hat{A}$ in $\Lambda^+_A$ is defined in direct analogy with the torsion of a connection in $\Lambda^1$. Concretely, consider the composition
\[
\Gamma(\Lambda^+_A) \stackrel{\nabla_{\hat{A}}}{\longrightarrow}
\Gamma(\Lambda^1 \otimes \Lambda^+_A)\stackrel{\sigma}{\longrightarrow}
\Gamma(\Lambda^3)
\]
where $\sigma$ is skew-symmetrisation. We write $\tau(\hat{A}) = \sigma \circ \nabla_{\hat{A}} - \diff$. This is a tensor, which we regard as lying in $ \Lambda^3 \otimes \Lambda^+$. 

If we treat $\Sigma_A$ as an element of $\Omega^2(M,E)$ then the covariant exterior derivative $\diff_\nabla \Sigma_A $ is an $E$-valued 3-form. One checks that $\tau(\hat{A})$ is identified with $\diff_A\Sigma_A$ via the isomorphism $1\otimes\Sigma_A \colon \Lambda^3\otimes E \to \Lambda^3\otimes \Lambda^+_A$. This explains the following definition.

\begin{definition}
Let $A$ be a definite connection. The 3-form $\diff_A \Sigma_A \in \Omega^3_A(M,E)$ is called \emph{the torsion of $A$}. When $\diff_A\Sigma_A = 0$ we say $A$ is \emph{torsion free}. 
\end{definition}

As the next Theorem explains, \emph{torsion-free definite connections correspond to Einstein metrics}. 

\begin{theorem}[Krasnov \cite{Krasnov}, cf.\ \cite{Fine}]
Let $A$ be a definite connection. When $A$ is torsion-free, the metric $g_A$ is Einstein, with $\Rm_+$ a definite endomorphism of $\Lambda^+$. If $A$ is positive definite, $\Rm_+>0$ and $R=12$. If $A$ is negative definite, $\Rm_+ <0$ and $R=-12$. Conversely all Einstein metrics satisfying such a curvature inequality arise this way.  
\end{theorem}

The point is that the $g_A$-Levi-Civita connection $\nabla$ on $\Lambda^+_A$ is the \emph{unique} torsion-free metric connection, so when $\diff_A\Sigma_A=0$ it follows that $\hat{A}= \nabla$. In particular, $\nabla$ is a self-dual instanton (since $A$ is) and this is equivalent to $g_A$ being Einstein. Moreover, the curvature of $\nabla$ is identified with $F_A$ which is definite and this gives the condition on $\Rm_+$. For the converse, given an Einstein metric $g$ with $\Rm_+$ definite, one takes $A$ to be the Levi-Civita connection on $\Lambda^+$. The fact that $\Rm_+$ is definite ensures that $A$ is a definite connection, the fact that $g$ is Einstein implies that $g_A = g$ and the fact that the Levi-Civita connection is torsion free implies $\diff_A \Sigma_A = 0$. 

\begin{remark}\label{why-tr-sqrt-3}
We close this section with an a~posteriori justification of the choice of volume form, via $\tr\sqrt{Q}_A = 3$. When $A$ is torsion-free, we have $\sqrt{Q}_A = \pm \Rm_+$ and so $\tr \sqrt{Q}_A = |R|/4$. If we are to have any chance of finding Einstein metrics we see we are forced to take this to be constant, which we thus ensure from the very beginning.  
\end{remark}

\section{The proof of Theorem~\ref{main-result}}\label{the-proof}

We now give the proof of Theorem~\ref{main-result}, that Poincaré--Einstein 4-manifolds with either $\Rm_+< 0$ or $\Rm_-<0$ are non-degenerate. The main step is to show the analogous statement for torsion-free definite connections; see Theorem~\ref{connections-non-degenerate} in \S\ref{connections-put-all-together} below. In \S\ref{complete-the-proof}, we explain how this implies the result for metrics.

\subsection{Asymptotically-hyperbolic definite connections}

We begin by determining the appropriate class of definite connections on the interior $M$ of a compact 4-manifold $\overline{M}$ with boundary $X$, which give rise to conformally compact metrics. In fact, the metrics $g_A$ will be automatically \emph{asymptotically hyperbolic}, which is to say that the sectional curvatures of $g_A$ tend to $-1$ at infinity. This class of definite connections is also treated, with different motivations in \cite{Fine-Herfray-Krasnov-Scarinci}. 

Let $\overline{E} \to \overline{M}$ be an oriented $\SO(3)$-bundle; we write $E \to M$ for the restriction of $\overline{E}$ to the interior. We fix an $\SO(3)$-connection $\overline{A}$ in $\overline{E}$ and a boundary defining function $\rho$. Recall also that we use the metric and fibre-wise volume form in $E$ to identify $\overline{E} \cong \overline{E}^* \cong\so(\overline{E}) \cong so(\overline{E})^*$.
\begin{definition}
An $\SO(3)$-connection $A$ in $E \to M$ is called a \emph{asymptotically hyperbolic} if: 
\begin{itemize}
\item
$\rho (\overline{A} - A) \in \Omega^1(M,E)$ extends to a section $a \in \Omega^1(\overline{M},\overline{E})$;
\item
moreover, at each boundary point $x\in X$, $a(x) \colon T_xX \to \overline{E}_x$ is an isomorphism.
\end{itemize}
Note that the definition does not depend on the choice of $\overline{A}$ or $\rho$. The isomorphism $TX \to \overline{E}|_X$ does depend on $\rho$, but is uniquely determined by $A$ up to multiplication by a nowhere vanishing function. We can use $a|_X$ to pull back the metric from $\overline{E}|_X$ to $TX$; the conformal class of this metric is independent of the choice of $\overline{A}$ or $\rho$ and so in this way an asymptotically hyperbolic connection $A$ determines a conformal structure on the boundary, which we denote by $\gamma_A$.
\end{definition}

\begin{lemma}
Let $A$ be an asymptotically hyperbolic connection. Then: 
\begin{enumerate}
\item  
$A$ is negative definite near the boundary.
\item 
The corresponding metric $g_A$ is asymptotically hyperbolic with conformal infinity equal to $\gamma_A$. 
\end{enumerate}
\end{lemma}
\begin{proof}
We begin by choosing a collar neighbourhood of the boundary which is well adapted to $A$. To do this, consider the  homomorphism $T\overline{M}|_X \to \overline{E}|_X$ given by the 1-form $\rho(\overline{A} - A)$ at points on the boundary. Since $A$ is asymptotically hyperbolic, the kernel of this map is a line sub-bundle $N\subset T\overline{M}|_X$, uniquely determined by $A$, with $T\overline{M}|_X \cong TX \oplus N$. Choose a vector field $\nu$ on $\overline{M}$, which spans $N$ at points of $X$ and points inwards. Flowing along $\nu$ defines a collar neighbourhood $U \cong X \times [0,\epsilon)$ of $X$ and the isomorphism $T\overline{M} \cong TX \oplus N$ extends in an obvious way to this collar, with $N = \left\langle \nu \right\rangle$. 

We next expand $A$ in this collar. Write $r \colon U \to [0,\epsilon)$ for the projection from the collar onto the $[0,\epsilon)$ factor. Over $U$ we have $A = \overline{A} + r^{-1}\alpha$ where $\alpha$ extends up to the boundary. Write $\alpha = \diff r \otimes b + c$ where $b=\alpha(\nu)$ and $c$ is an $E$-valued 1-form which vanishes on $N$. By construction, $b$ vanishes on the boundary. Meanwhile, on the boundary $c \colon TX \to \overline{E}|_X$ is an isomorphism. It follows that $\overline{E}$ is trivial over $X$ (since $TX$ is) and hence over the whole collar. We choose an orthonormal frame $e_i$ of $\overline{E}$ over the collar and write $c = \sum c_i(r) \otimes e_i$ for a path of 1-forms $c_i(r)$ on $X$. To ease the notation we write $\phi_i =c_i(0)$. The $\phi_i$ are a coframe for $T^*X$ and we define a metric on $X$ by declaring the $\phi_i$ to be orthonormal. This metric represents the conformal class $\gamma_A$ determined by $A$.   

Consider now the asymptotically hyperbolic metric $g = r^{-2}\left(\diff r^2 + \phi_1^2 + \phi_2^2 + \phi^2_3\right)$ defined on the collar, with conformal infinity $\gamma_A$. We will show that $A$ is definite near the boundary and that the metric it induces is $g_A = g + O(r)$, where $O(r)$ is measured with respect to $g$.

The curvature of $A$ is given by
\begin{equation*}
F_A 
  =
    -r^{-2} \left(\diff r \wedge \alpha - [\alpha \wedge \alpha]\right)
    + 
    r^{-1}\diff_{\overline{A}}\alpha
    +
    F_{\overline{A}}
\end{equation*}
Recalling that $b$ vanishes on the boundary, we see that $F_A = \sum F_i \otimes e_i$ where
\begin{align}
F_1
  &=
    r^{-2}\left(\phi_1 \wedge \diff r + \phi_2 \wedge \phi_3\right) + O(r),\label{F1}\\
F_2
  &=
    r^{-2}\left(\phi_2 \wedge \diff r + \phi_3 \wedge \phi_1\right) + O(r), \label{F2}\\
F_3
  &=
    r^{-2}\left(\phi_3 \wedge \diff r + \phi_1 \wedge \phi_2\right) + O(r) \label{F3}.
\end{align}
Here the first term in each $F_i$ is unit-length with respect to the metric $g$ and $O(r)$ is measured with respect to $g$.

Let $\mu = r^{-4}\phi_1\wedge\phi_2 \wedge \phi_3 \wedge \diff r$ be the volume form of $g$, then 
\[
Q(\mu)_{ij} = \delta_{ij} + O(r).
\]
It follows that $A$ is definite near the boundary and one can check that the sign of the triple $F_i$ is negative. Moreover the actual volume form determined by $A$ is $\mu_A = \mu + O(r)$. Meanwhile, the leading term in each curvature form \eqref{F1}, \eqref{F2}, \eqref{F3} is self-dual with respect to $g$. It follows that $g_A$ and $g$ agree to $O(r)$. In particular, $g_A$ is also asymptotically hyperbolic and it has conformal infinity $\gamma_A$.
\end{proof}

\subsection{Infinitesimal deformations of torsion-free definite connections}
\label{connections-put-all-together}

Let $A$ be a torsion-fee asymptotically-hyperbolic definite connection and $\delta A = a$ an infinitesimal deformation preserving the torsion-free condtion. In this subsection, we prove Theorem \ref{connections-non-degenerate}: that if $a \in L^2$ then $a$ is pure gauge.  The strategy of the proof is as follows.
\begin{itemize}
\item
We start with Lemma~\ref{integral-condition}, which gives a technical integral condition which $a$ must satisfy. 
\item
Definition~\ref{spin-Coulomb-gauge} introduces a judicious choice of gauge, combining gauge fixing for both vertical and horizontal infinitesimal gauge transformations. In Proposition~\ref{effect-of-gauge} we show that when $a$ is fully gauge fixed, the integral condition of Lemma~\ref{integral-condition} actually forces $a$ to vanish identically.
\item
In Proposition~\ref{gauge-fixing} we show that any $a \in L^2$ can be written as a sum of a piece which is fully gauge fixed and a piece which is pure gauge. From here we can complete the proof of Theorem~\ref{connections-non-degenerate}.
\end{itemize}
We begin by fixing our notation. Let $D_A$ denote the linearisation at $A$ of the map $A \mapsto \diff_A \Sigma_A$. We are interested in $a \in C^\infty(M,\Lambda^1\otimes E)$ for which $D_A(a)=0$. 

We work in a local oriented orthonormal frame $e_1, e_2, e_3$ for $E$. We write the dual frame as $e^i$ and we denote by $\hat{e}^i$ the generator of positive rotations about~$e_i$. We use the summation convention throughout. We write the connection $A$ locally as
\begin{equation}\label{A-local}
\nabla_A (e_i) = \epsilon_{ijk}A^j \otimes e^k
\end{equation}
for a triple of 1-forms $A^j$. If $a = a^i \otimes e_i$ is an $E$-valued 1-form then our conventions give that $\diff_A a = (\diff_A a)^i\otimes e_i$ where 
\[
(\diff_A a)^i = \diff a^i - \epsilon^{i}_{\phantom{i}jk}A^j \wedge a^k
\]
The curvature is determined locally by the triple of 2-forms
\begin{equation}\label{F-local}
F^i = \diff A^i -\frac{1}{2}\epsilon^{i}_{\phantom{i}jk}A^j \wedge A^k
\end{equation}
One can check that $(\diff_A)^2(e_i) = \epsilon_{ij}^{\phantom{ij}k}F^j \otimes e_k$. 
It is convenient to write
\[
\left(\sqrt{Q}_A\right)^{ij} = -\Psi^{ij} + \delta^{ij}
\]
where $\Psi$ is a trace-free symmetric matrix. (Recall that $\tr \sqrt{Q} = 3$ which is why the trace part here is $\delta^{ij}$.) We write $\Sigma_A = \Sigma_i \otimes e^i$. The equation $F_A=-\sqrt{Q}_A \Sigma_A$ (the sign appears because we are dealing with a negative definite connection) has the following local expression:
\begin{equation}\label{F-from-Sigma}
F^i 
	= 
		\left( \Psi - \delta\right) ^{ij}
		\Sigma_j
\end{equation}
To ease notation we write $P = (\Psi - \delta)^{-1}$. (The inverse exists because $A$ is definite.) 

\begin{lemma}\label{integral-condition}
Let $A$ be a torsion-free asymptotically-hyperbolic definite connection and $\delta A = a \in C^{\infty}(M,\Lambda^1 \otimes E)$ an infinitesimal deformation of $A$ which preserves the torsion-free condition, i.e.\ with $D_A(a)=0$, and which also lies in $W^{2,1}$. Then,
\[
\int_M \left[P_{ik}\left(
(\diff_A a)^k - \phi^{kj}\Sigma_j\right)
\wedge (\diff_A a)^i
- \epsilon^i_{\phantom{i}jk}a^j \wedge a^k \wedge \Sigma_i\right]
	=
		0
\]	
where $\phi = \delta \Psi$ denotes the corresponding infinitesimal change in $\Psi$ under $\delta A = a$. 
\end{lemma}
\begin{proof}
Given an asymptotically hyperbolic connection $A$ and $b \in W^{2,1}(M,\Lambda^1\otimes E)$, consider the following quantity
\begin{equation}\label{key-quantity}
q_A(b) = \int_M \Sigma_i\wedge (\diff_A b)^i
\end{equation}
Integrating by parts to put the derivative $\diff_A$ on $\Sigma_A$ we see that when $A$ is torsion-free, $q_A(b)=0$ for all $b$.. Write $h_A(a,b)$ for the derivative of $q_A(b)$ with respect to $A$ in the direction $\delta A^i = a^i$.  

To compute $h_A(a,b)$ we first denote the corresponding infinitesimal change in $\Sigma$ by $\delta \Sigma_i = \sigma_i$ and in $\Psi$ by $\delta \Psi^{ij} = \phi^{ij}$.
Differentiating \eqref{F-local} and \eqref{F-from-Sigma} we see that
\[
(\diff_A a)^i -
	\left(\Psi - \delta\right)^{ij} \sigma_j - \phi^{ij}\Sigma_j
	=
		0
\]
This rearranges to give
\begin{equation}\label{sigma-from-a}
\sigma_i 
	= 
		P_{ik}\left(
		(\diff_A a)^k - \phi^{kj}\Sigma_j
		\right).
\end{equation}	
Differentiating formula \eqref{key-quantity} we see that for any $a, b \in W^{2,1}$, we have
\begin{equation}\label{h-definition}
h_A(a,b)	
=
\int_M \left[P_{ik}\left(
(\diff_A a)^k - \phi^{kj}\Sigma_j\right)
\wedge (\diff_A b)^i
- \epsilon^i_{\phantom{i}jk}a^j \wedge b^k \wedge \Sigma_i\right].
\end{equation}
Since in our case $\delta A=a$ preserves the condition $\diff_A \Sigma_A =0$ we must have $h_A(a,b)=0$ for all $b \in W^{2,1}$. In particular, $h_A(a,a)=0$ which is the identity in the statement of the Lemma.
\end{proof}

%\begin{remark}
%One can interpret $q_A(b)$ in \eqref{key-quantity} as a 1-form on the space of definite connections. Integration by parts shows that	the zeros of this 1-form are precisely the torsion-free connections. In the compact setting this 1-form is exact, it is the differential of the \emph{Krasnov action}, which is simply the total volume $S(A) = \pm\int \mu_A$, with sign agreeing with that of the connection. The quantity $h_A(a,b)$ is related to, but not equal to, the Hessian of $S$ at $A$. The Krasnov action and its Hessian was used in \cite{FKS} to prove local rigidity of certain \emph{compact} Einstein 4-manifolds. Poincaré--Einstein metrics have infinite volume and so the Krasnov action needs to be ``renormalised'' in this setting. See \cite{Fine-Herfray-Krasnov-Scarinci} for more in this direction.
%\end{remark}

%We next introduce a convient gauge-fixing condition.

\begin{definition}\label{spin-Coulomb-gauge}
Let $a \in \Gamma(M,\Lambda^1\otimes E)$ and $A$ be a torsion-free definite connection.
\begin{itemize}
\item
We say that $a$ is in \emph{vertical gauge with respect to $A$} if
\begin{equation}
\Sigma_i \wedge a^i = 0.
\label{spin-gauge}	
\end{equation}
\item
We say that \emph{$a$ is in horizontal gauge with respect to $A$} if, in addition,
\begin{equation}
\epsilon^{ij}_{\phantom{ij}k}\Sigma_j \wedge (\diff_A a)^k = 0.
\label{Coulomb-gauge}	
\end{equation}
\item
When $a$ is in both vertical and horiztonal gauge with respect to $A$ we say that \emph{$a$ is fully gauge fixed with respect to $A$}.
\end{itemize}
\end{definition}

\begin{proposition}\label{effect-of-gauge}
Let $A$ be a torsion-free asymptotically-hyperbolic definite connection and $\delta A =a$ an infinitesimial deformation of $A$ which lies in $W^{2,1}$ and which preserves the torsion-free condition, i.e.\ with $D_A(a)=0$. If $a$ is fully gauge fixed with respect to $A$ then in fact $a=0$.
\end{proposition}

To prove this we begin with a couple of Lemmas. The first appears in \cite{FKS} but the proof is so short we repeat it here for convenience.

\begin{lemma}\label{star-lemma}
Let $A$ be torsion-free and $a$ be in vertical gauge with respect to $A$. Then
\begin{align}
* a^i 
	&= 
		\epsilon^{ij}_{\phantom{ij}k}\Sigma_j \wedge a^k,
		\label{Hodge-star-in-spin-gauge}\\
(\diff_A^*a)^i
	&=
		- * \left(
		\epsilon^{ij}_{\phantom{ij}k}
		\Sigma_j \wedge (\diff_A a)^k
		\right).
		\label{dstar-in-spin-gauge}
\end{align}
In particular, when $a$ is in vertical gauge, condition \eqref{Coulomb-gauge} agrees with the  usual definition $\diff_A^*a=0$ of Coulomb gauge.
\end{lemma}

\begin{proof}
Recall the almost complex structures $J_i$ from \eqref{J-definition}, which were used in determining the sign of a definite connection. The vertical gauge condition \eqref{spin-gauge} gives $J_ia^i=0$ which, by the quaternion relations, is equivalent to $a^i = - \epsilon^{ij}_{\phantom{ij}k}J_j a^k$. Applying the Hodge star $*$ and using the fact that $*^2=-1$ on 3-forms gives \eqref{Hodge-star-in-spin-gauge}. Now applying $\diff_A$ and using the fact that $\diff_A\Sigma_A=0$  and $\diff_A^* = -*\diff_A *$ gives \eqref{dstar-in-spin-gauge}.
\end{proof}

%The second Lemma we need can also be extracted from the calculations in \cite{FKS}, but we give a more direct argument here.

\begin{lemma}\label{dminus}
Let $A$ be a torsion-free negative-definite connection. Let $\delta A = a$ be fully gauge fixed with respect to $A$. As above, write $\phi = \delta \Psi$ for the corresponding infinitesimal change in $\Psi$. Then
\[
(\diff_A^- a)^i = (\diff_A a)^i - \phi^{ij}\Sigma_j
\]	
\end{lemma}

\begin{proof}
Differentiating $F_A = (\Psi-\delta)\Sigma_A$ shows that the infinitesimal change $\delta \Sigma_i = \sigma_i$ in $\Sigma_i$ satisfies
\begin{equation}
(\diff_A a)^i 
	= 
		\phi^{ij}\Sigma_j
		+
		\left(\Psi^{ij} - \delta^{ij}\right)\sigma_j.
\label{da-self-dual}
\end{equation}
We will show that $\sigma_i^+=0$ from which it follows that $\diff_A^+a = \phi\Sigma$ and so $\diff_Aa - \phi\Sigma = \diff_A^-a$ as claimed.

We use the orthonormal frame $\Sigma_i$ of $\Lambda^+_A$ to write $\sigma_i^+ = Z_i^{\phantom{i}j}\Sigma_j$ for some 3-by-3 matrix-valued function $Z$. Similarly we write $(\diff_A^+a)^i = N^{ij}\Sigma_j$. Taking the self-dual part of \eqref{da-self-dual} gives
\begin{equation}\label{N=phi+YZ}
N^{ij} = \phi^{ij} + (\Psi^{ik} - \delta^{ik})Z_{k}^{\phantom{k}j}.
\end{equation}
We will use this equation to prove that $Z=0$.
 
By orthonormality of the $\Sigma_i$,
\[
\Sigma_i \wedge \Sigma_j = \frac{1}{3}(\Sigma_k \wedge \Sigma_k)\delta_{ij}.
\]
Differentiating this implies that the symmetric trace-free part of $Z$ vanishes. Meanwhile, the skew-symmetric and trace parts of $N$ vanish. The fact that the skew part is zero is precisely the second gauge condition~\eqref{Coulomb-gauge}; meanwhile taking $\diff_A$ of the first gauge condition~\eqref{spin-gauge} and using $\diff_A \Sigma_A=0$ shows that $\tr N =0$. 

Write $Z = S+c\Id$ where $S$ is skew-symmetric. Taking the trace of~\eqref{N=phi+YZ}, and using $\tr N=0=\tr \phi$, gives $c = \frac{1}{3}\tr((\Psi-\delta)S)$ so to prove $Z=0$ it suffices to show that $S=0$. Next, since $N$ is symmetric,~\eqref{N=phi+YZ} implies that $(\Psi - \delta)S$ is also symmetric or, equivalently, that
\begin{equation}\label{S-Y-anticommute}
(\Psi-\delta)^{-1}S(\Psi -\delta)+S =0.
\end{equation}
Equation~\eqref{S-Y-anticommute} shows that $\ker S$ is perseved by $(\Psi - \delta)$. Since $\Psi- \delta$ is symmetric, it also preserves the orthogonal complement $(\ker S)^\perp$. If this complement is non-trivial we can find $w \in (\ker S)^\perp$ which is an eigenvector of $\Psi- \delta$ with eigenvalue $\lambda <0$ (since $\Psi - \delta <0$). Applying~\eqref{S-Y-anticommute} to $w$ shows that $(\Psi -\delta) S(w) =  - \lambda S(w)$. This means that $S(w)$ is an eigenvector of $\Psi-\delta$ with a positive eigenvalue, but this contradicts the fact that $\Psi-\delta < 0$. This forces $S(w)=0$ which in turn contradicts $w \in( \ker S)^\perp$. So $S$ vanishes as claimed.
\end{proof}

We can now prove Proposition~\ref{effect-of-gauge}.

\begin{proof}[Proof of Proposition~\ref{effect-of-gauge}.]
Lemma~\ref{integral-condition} shows that when $\delta A = a$ preserves the torsion-free condition then $h_A(a,a)=0$, i.e.\
\[
\int_M \left[P_{ij}\left(
(\diff_A a)^j - \phi^{ij}\Sigma_j\right)
\wedge (\diff_A a)^i
- \epsilon^i_{\phantom{i}jk}a^j \wedge a^k \wedge \Sigma_i\right]
	=
		0
\]
Now, when $a$ is fully gauge fixed, Lemma~\ref{dminus} shows that $\diff_A -\phi\Sigma=\diff_A^-a$. Meanwhile, by Lemma~\ref{star-lemma}, the second term here is
\[
-
\epsilon^i_{\phantom{i}jk}a^j \wedge a^k \wedge \Sigma_i
	=
		a^j \wedge * a^j
	=
		|a|^2\mu_A.
\] 
So
\[
\int_M 
	\left[
	P_{ij} (\diff_A a^-)^i\wedge (\diff_Aa^{-})^j
	+ 
	|a|^2 \mu_A
	\right]
		=
			0.
\]
For anti-self-dual 2-forms $\theta,\chi$ we have $\theta \wedge \chi = - \langle \theta, \chi \rangle \mu_A$ and so we find
\[
\int_M \left[ 
	- P_{ij}
	\langle (\diff_A^-a)^i, (\diff_A^-a)^j\rangle
	+
	|a|^2\right] \mu_A
		=
			0.
\]
Since $P<0$ both terms of the integrand are pointwise non-negative and so both must vanish, giving $a=0$.
\end{proof}

We next show that our gauge-fixing condition gives a complement to the infinitesimal gauge action $R_A$ at $A$. Recalling the formula for this action (Lemma~\ref{infinitesimal-gauge-action}) set
\begin{align*}
V_A & = \{ \diff_A u + \iota_v F_A : u \in W^{2,2}(M,E),\, v \in W^{2,1}(M,TM)\} \\
W_A & = \{ a\in W^{2,1}(M,\Lambda^1\otimes E) : a\text{ is fully gauge fixed with respect to }A\}.
\end{align*}

\begin{proposition}\label{gauge-fixing}
Let $A$ be a torsion-free asymptotically-hyperbolic definite connection. There is a direct sum decomposition: 
\[
W^{2,1}(M,\Lambda^1 \otimes E) = V_A \oplus W_A.
\]
\end{proposition}

We break the proof of this up into a series of Lemmas. This result is the asymptotically hyperbolic analogue of Proposition~3.6 from \cite{FKS}, which deals with the compact case. Accordingly we will go quickly over the parts which are proved in detail there and focus here on the parts of the proof which are new, and which are based on the $0$-calculus. 

To begin, recall the definition~\eqref{J-definition} of the triple of almost quaternionic structures $J_i$, defined (locally at least) by a definite connection. We write
\[
p \colon \Lambda^1 \otimes E \to \Lambda^1, 
\quad p(a) = J_ia^i
\]
Whilst the individual $J_i$ are only defined locally, $p$ is globally defined. A section $a$ is in vertical gauge with respect to $A$ if and only if $p(a)=0$. 

Next, write
\[
q \colon TM \to \Lambda^1\otimes E,
\quad
q(v) = \iota_v F_A
\]
To prove Proposition~\ref{gauge-fixing} we must show that for any $a\in W^{2,1}(M,\Lambda^1\otimes E)$ there is a unique $u \in W^{2,2}(M,E)$ and $v \in W^{2,1}(M,TM)$ solving the pair of equations
\begin{align}
p\left( a + \diff_A u + q(v) \right)
  &=
    0,\label{GF1}\\
\diff_A^*\left( a + \diff_A u + q(v) \right)
  &=
    0.\label{GF2}
\end{align}
One checks that $p\circ q = \colon TM \to \Lambda^1$ is equal to a non-zero multiple of the metric isomorphism between vectors and covectors (see Lemma~3.7 of \cite{FKS} for details). Then $v$ can be recovered from $a$ and $u$ via~\eqref{GF1}:
\[
v = - (p\circ q)^{-1}\circ p \left( \diff_A u +a\right)
\]
From here one sees that Proposition~\ref{gauge-fixing} is equivalent to the assertion that the following second-order equation has a unique solution  $u \in W^{2,2}$:
\begin{equation}
\diff_A^* \Pi_A \diff_A u =  - \diff_A^* \Pi_A a, 
\label{GF-full} 
\end{equation}
where 
\[
\Pi_A = 1 - q \circ(p\circ q)^{-1} \circ p
\]
is the projection operator with $\im \Pi_A= \ker p$ and $\ker \Pi_A = \im q$. 

We will solve~\eqref{GF-full} via the theory of $0$-elliptic operators. Lemma~3.9 of~\cite{FKS} shows that $\diff_A^*\Pi_A \diff_A$ is elliptic at interior points of $M$. Meanwhile, the normal operator of $\diff_A^*\Pi_A\diff_A$ is the analogous operator determined by the Levi-Civita connection acting on the bundle of self-dual 2-forms $\Lambda^+ \to \H^4$. This follows from the reasoning in Remark~\ref{geometric-normal-operators}. The point is that the torsion-free connection $A$ is identified with the Levi-Civita connection of $g_A$ on $\Lambda^+_A$. Since $g_A$ is asymptotically hyperbolic, in the limit used in defining the normal operator the connection $A$ converges to the Levi-Civita connection of $\Lambda^+ \to \H^4$. 

We now analyse this normal operator. We write $C$ for the Levi-Civita connection in $\Lambda^+ \to \H^4$ and denote the corresponding operator by $\diff^*_C\Pi_C\diff_C$.

\begin{lemma}
The normal operator $\diff_C^*\Pi_C\diff_C$ has the following properties:
\begin{itemize}
\item
Its indicial roots are $-1$ and $4$.
\item
The map $\diff_C^*\Pi_C \diff_C \colon W^{2,2}(\H^4, \Lambda^+) \to L^2(\H^4,\Lambda^+)$ is an isomorphism.
\end{itemize}
\end{lemma}
\begin{proof}

Let $u\in \ker \Pi_C\diff_C$, so that there is a vector field $v$ for which $\diff_Cu= \iota_v F_C$. The $C$-horizontal lift $\hat{v}$ of $v$ to $\Lambda^+$ has an infinitesimal action on $C$ which is equal to $\iota_vF_C$ and so to $\diff_C u$, which is pure \emph{vertical} gauge. This means that the metric $g_C$ doesn't change, in other words that $v$ is a Killing field for the hyperbolic metric. The Killing fields of $\H^4$ all have the property that $|v|\sim \rho^{-1}$ (where $\rho$ is a boundary defining function) and so $\Pi_C\diff_C$ has a single indicial root, $-1$. 

Next one checks that, in this model case, $\ker p$ and $\im q$ are orthogonal, so that $\Pi_C$ is self-adjoint. It follows that 
\[
\diff^*_C \Pi_C \diff_C
  =
    \left(\Pi_C\diff_C\right)^*(\Pi_C \diff_C)
\]
and so has indicial roots $-1, 4$ (by the symmetry of indicial roots for adjoints around the value $3$). 

The relation with Killing fields shows that $\diff^*_C\Pi_C \diff_C \colon W^{2,2} \to L^2$ is also injective: if $u \in W^{2,2}$ is in the kernel of $\diff^*_C\Pi_C\diff_C$ then  $0= \left\langle \diff_C^*\Pi_C\diff_C u , u  \right\rangle_{L^2} = \| \Pi_C\diff_C u\|^2_{L^2}$ and so $\Pi_C\diff_C u = 0$. This means that $\diff_C u = \iota_vF_C$ for some Killing field $v$, but there are no Killing fields in $L^2$. So $\diff_C u =0$ meaning $u$ is parallel. But since $u \to 0 $ at infinity, the only possibility is that $u =0$ identically. 

To complete the proof we will show that there is a constant $C$ such that for all $u \in W^{2,2}$, 
\begin{equation}
\| u \|_{W^{2,2}} \leq C \| \diff_C^*\Pi_C\diff_C u\|_{L^{2}}. 
\label{coercive-estimate}  
\end{equation}
From here it follows that the range of $\diff_C^*\Pi_C\diff_C$ is closed. Since it is self-adjoint and injective, if then follows that it is also surjective.

The calculations are long, but not especially difficult and so we suppress a lot of the details. We use half-space coordinates $\rho >0$ and $y^1,y^2,y^3 \in \R$, in which the hyperbolic metric is $g_C = \rho^{-1}(\diff \rho^2 + \diff y^2)$. We use $\alpha^0 = \rho^{-1}\diff\rho$ and $\alpha^i = \rho^{-1} \diff y^i$ as an orthonormal coframe, and $\alpha^{0123}$ as the volume form, to fix the orientation. We trivialise $\Lambda^+$ with the orthonormal frame 
\[
e_1 = \frac{1}{\sqrt{2}}\left(\alpha^{01}+\alpha^{23}\right)
\quad
e_2 = \frac{1}{\sqrt{2}}\left(\alpha^{02}+\alpha^{31}\right)
\quad
e_3 = \frac{1}{\sqrt{2}}\left(\alpha^{03}+\alpha^{12}\right)
\]
We write sections of $\Lambda^+$ as $u = u^i e_i$ and so forth. One checks that the Levi-Civita connection is given in this frame by
\[
\diff_C u
  = 
    \diff u
    + 
    \begin{pmatrix}
    0 & \alpha^3 &  -\alpha^2  \\
    -\alpha^3 & 0 & \alpha^1 \\
    -\alpha^2 & -\alpha^1 & 0
    \end{pmatrix}
    \begin{pmatrix} u^1\\u^2\\u^3\end{pmatrix}
\]
From here one can compute the curvature 2-forms and hence the corresponding triple of almost quaternion structures, which are given by
\[
\begin{array}{l|l|l}
J_1(\alpha^0) = \alpha^1 &
  J_2(\alpha^0) =  \alpha^2 &
      J_3(\alpha^0) =  \alpha^3\\
J_1(\alpha^1) = -\alpha^0 &  
  J_2(\alpha^1) = -\alpha^3 &
    J_3(\alpha^1) =  \alpha^2\\
J_1(\alpha^2) =   \alpha^3 &
  J_2(\alpha^2) = -\alpha^0 &
     J_3(\alpha^2) = -\alpha^1 \\
J_1(\alpha^3) = -\alpha^2 &
  J_2(\alpha^3) =  \alpha^1&
    J_3(\alpha^3) = -\alpha^0
\end{array}
\]
One then checks that, for $a = a^i \otimes e_i$,  
\[
\Pi_C(a)
  =
   \frac{1}{3}
   \begin{pmatrix}
    2a^1 + J_3a^2 - J_2a^3\\
    2a^2 + J_1a^3 - J_3a^1\\
    2a^3 + J_2a^1 - J_1a^2  
    \end{pmatrix}
\]
(From here one can verify directly that $\Pi_C^*= \Pi_C$.) This then leads to the following formula for $\Pi_C\diff_C$:
\[
\Pi_C \diff_C (u)
  =
    \frac{1}{3}
    \begin{pmatrix}
      2(\diff u^1 + u^1\alpha^0) 
      + 
      (J_3\diff u^2 + u^2 \alpha_3) 
      - 
      (J_2\diff u^3 + u^3 \alpha_2)\\
      2(\diff u^2 + u^2\alpha^0) 
      + 
      (J_1\diff u^3 + u^3 \alpha_1) 
      - 
      (J_3\diff u^1 + u^1 \alpha_3)\\
      2(\diff u^3 + u^3\alpha^0) 
      + 
      (J_2\diff u^1 + u^1 \alpha_2) 
      - 
      (J_1\diff u^2 + u^2 \alpha_1)
      \end{pmatrix}  
\]
(And from here one can verify directly that $\Pi_C\diff_C$ has $-1$ as its only indicial root.) Further computations lead to
\begin{equation}\label{GF-normal-op}
\frac{3}{2} \diff^*_C \Pi_C\diff_C(u)
=
(\diff^* \diff + 4)u
+
Ru
\end{equation}
where
\[
Ru=
\frac{1}{2}
\begin{pmatrix}
  0 & - \rho \del_3 & \rho\del_2 \\
  \rho \del_3 & 0 & - \rho \del_1 \\
  -\rho \del_2 & \rho \del_1 & 0  
  \end{pmatrix}
\begin{pmatrix} u^1\\u^2\\u^3\end{pmatrix}
\]
(Again, that this confirms our previous observation that the indicial roots are $-1, 4$.) 

We now reach the pay-off. One checks directly that $\| Ru\|_{L^2}^2 \leq \frac{1}{2}\|\diff u\|^2_{L^2}$.
Now from~\eqref{GF-normal-op}, we have
\begin{align*}
\frac{9}{4} \| \diff^*_C\Pi_C\diff_C(u)\|^2_{L^2}
  &\geq
    \| (\diff^*\diff + 4)u\|^2_{L^2} 
    - 
    2\| Ru\|_{L^2}\|(\diff^*\diff + 4)u\|_{L^2},\\
  & \geq  
    \frac{1}{2}\|(\diff^*\diff + 4)u\|^2_{L^2}
    -
    2\| R u\|^2_{L^2},\\
  & \geq 
    \frac{1}{2}\| \diff^*\diff u\|^2_{L^2}
    +
    3 \| \diff u\|^2_{L^2}
    +
    16 \|u\|^2_{L^2},\\
    &\geq
    c\|u\|^2_{W^{2,2}},
\end{align*}
where in the last line we used the fact that the hyperbolic Laplacian acting on functions is an isomorphism $\diff^*\diff \colon W^{2,2} \to L^2$. This establishes~\eqref{coercive-estimate}, completing the proof. 
\end{proof}

The $0$-calculus (see Theorem~\ref{mazzeo-fredholm} above) now gives the following corollary.

\begin{corollary}
Let $A$ be a torsion-free asymptotically-hyperbolic definite connection. Then $\diff_A^*\Pi_A\diff_A \colon W^{2,2}(M, E) \to L^2(M,E)$ is a Fredholm map.
\end{corollary}

With this in hand we can now complete the gauge-fixing argument.
\begin{proof}[Proof of Proposition~\ref{gauge-fixing}]
Following the discussion around equation~\eqref{GF-full}, to prove the result we must show that 
\[
\diff_A^*\Pi_A\diff_A \colon W^{2,2}(M,E) \to L^2(M,E)
\]
is an isomorphism. We already know that it is Fredholm. Moreover, this map is also injective. To see this note that if $u \in \ker \diff_A^*\Pi_A\diff_A$ then, just as in the model case of $\Lambda^+ \to \H^4$, it follows that $\diff_A u = \iota_v F_A$ for a Killing field $v$. But the fact that $\diff_A u \in L^{2}$ means that $v \in L^{2}$. Now we use the Bochner formula for a Killing field:
\[
\frac{1}{2}\Delta |v|^2 = |\nabla v|^2 - \Ric(v,v) = |\nabla v|^2 + 3 |v|^2
\]
(since $\Ric = - 3g$). Since $v \in L^2$, $|v|$ must attain a local maximum at an internal point of $M$. The maximum principle then tells us that $v=0$ at that point and so $v$ vanishes identically. As a consequence, $\diff_A u =0$, but the only parallel section of $E$ which is also in $L^2$ is $u=0$.

To complete the proof we will show that $\diff_A^*\Pi_A\diff_A$ has index zero. If it were self-adjoint, this would be automatic, but since in general $\Pi_A$ is not self-adjoint, $\diff_A^*\Pi_A\diff_A$ is not either. Instead, we connect $\Pi_A$ to a self-adjoint projection, using the same trick as in \cite{FKS}. Recall that $F^i = (\Psi^{ij} - \delta^{ij})\Sigma_j$. Consider the path $F^i_t = (t\Psi^{ij} - \delta^{ij})\Sigma_j$ of $E$-valued 2-forms. This gives a path of maps $q_t(v) = \iota_v F_t$ and so a path of projection operators, $\Pi_t = 1 - q_t \circ (p \circ q_t)^{-1} \circ p$, projecting onto $\ker p$ against $\im q_t$. We have $\Pi_1 = \Pi_A$ whilst $\Pi_0$ is genuinely self-adjoint, since $\im q_0$ is the orthogonal complement to $\ker p$. It is shown in \cite{FKS} that the corresponding operators $D_t = \diff_A^*\Pi_t \diff_A$ are elliptic at interior points. Meanwhile, the metric $g_A$ is asymptotically hyperbolic, which implies that $\Psi = O(\rho)$ (where $\rho$ is a boundary defining function). This shows that changing $t$ does not affect the normal operator of $D_t$, which is again $\diff_C^*\Pi_C\diff_C$. It follows from Theorem~\ref{mazzeo-fredholm} that $D_t \colon W^{2,2} \to L^2$ is a path of Fredholm operators and so has constant index. Since $D_0$ is self-adjoint, this common index must be zero. In particular, the index of $D_1 = \diff_A^*\Pi_A\diff_A$ is seen to be zero. 
\end{proof}

\begin{theorem}\label{connections-non-degenerate}
Let $A$ be a torsion-free asymptotically hyperbolic definite connection and let $\delta A =a$ be an infinitesimal deformation of $A$ which preserves the torsion-free condition, i.e.\ with $D_A(a)=0$, and which lies in $W^{2,1}$. Then $a$ is pure gauge, i.e.\ there exists $u \in W^{2,2}(M,E)$ and $v \in W^{2,1}(M,TM)$ such that $a = \diff_A u + \iota_v F_A$. In particular the infinitesimal change in the corresponding Einstein metric is given by a Lie derivative: $\delta (g_A) = \L_v(g_A)$.
\end{theorem}
\begin{proof}
By Proposition~\ref{gauge-fixing} we can write $a = b + c$ where $b$ is in fully gauge fixed with respect to $A$ and $c = \diff_A u + \iota_vF_A$ for some $u,v$. The torsion-free condition $\diff_A \Sigma_A=0$ is gauge invariant and so $D_A(c)=0$. Since $D_A(a)=0$ we have that $D_A(b) = 0$ as well. Now Proposition~\ref{effect-of-gauge} gives that $b=0$ and so $a = c$ as required.
\end{proof}

\subsection{Completing the proof of non-degeneracy}\label{complete-the-proof}

We are now in a position to give the proof of the main result, Theorem~\ref{main-result}. \begin{proof}[Proof of Theorem~\ref{main-result}]
Let $(M,g)$ be an oriented Poincaré--Einstein 4-manifold with $\Rm_+ <0$ and let $h$ be an $L^2$-solution of $L_g(h)=0$. We will prove that $h=0$. (It is enough to consider $\Rm_+<0$; for the case $\Rm_-<0$ one simply reverses the orientation.)

Let $A$ denote the Levi-Civita connection of $g$ on $\Lambda^+$. The curvature condition $\Rm_+<0$ and the fact that $g$ is Einstein imply that $A$ is a torsion-free asymptotically-hyperbolic definite connection with $g_A=g$. We claim that there is an infinitesimal deformation $a \in W^{2,1}$ of $A$ for which $\delta(g_A) = h$ and $D_A(a)=0$ (i.e.\ $a$ preserves the torsion-free condition). 

The construction of $a$ follows the path laid out in \S4 of \cite{FKS}. We briefly summarise the argument here.  Let $\hat{g}(t)$ be a path of asymptotically hyperbolic metrics on $M$ with $\hat{g}(0) = g$ and $\hat{g}'(0) = h$. Choose a path of isometries $\hat{\Sigma}(t) \colon E \to \Lambda^+(\hat{g}(t))$ with $\hat\Sigma(0) = \Sigma_A$. Write $C(t)$ for the Levi-Civita connection of $\hat{g}(t)$ on $\Lambda^+(\hat{g}(t))$ and pull this back via $\hat{\Sigma}(t)$ to give a path $A(t) =\hat{\Sigma}(t)^*C(t)$ of $\SO(3)$-connections in $E$. We have $A(0) = A$. For small $t$, $A(t)$ is a definite connection and so determines a metric $g(t) = g_{A(t)}$. 

It is important to note that in general $g(t) \neq \hat{g}(t)$. At this point we invoke Proposition~\ref{gauge-fixing-does-job} which ensures that $D_g(h)=0$ (where $D_g$ is the linearisation of $g \mapsto \Ric(g)+3g$). Now, \S4 of \cite{FKS} proves that, because $D_g(h)=0$, we  have $g'(0) = \hat{g}'(0) = h$.  Write $A'(0) = a$; \S4 of \cite{FKS} also shows that $D_A(a)=0$. So $a$ is  the section of $\Lambda^1 \otimes E$ we are looking for. 

We now check that $a \in W^{2,1}$. Since $h \in L^2$ and $L_g(h)=0$ it follows from the fact that $L_g$ is $0$-elliptic that $h \in W^{2,2}$. (See for example Lemma~4.8 in \cite{Lee}.) This means that $C'(0) \in W^{2,1}$ and that we can also choose $\hat{\Sigma}$ so that $\hat{\Sigma}'(0) \in W^{2,1}$.  From here it follows that $A'(0) \in W^{2,1}$ as required.

Since $D_A(a)=0$ and $a \in W^{2,1}$, Theorem~\ref{connections-non-degenerate} now ensures that there is a vector field $v \in W^{2,1}$ on $M$ for which $h = \L_v(g)$. By Proposition~\ref{gauge-fixing-does-job}, $B_g(h)=0$ which, by the definition~\eqref{Bianchi-operator} of $B_g$, means that
\[
\div_g (\L_v(g)) + \frac{1}{2}\diff (\tr \L_vg)=0
\] 
Integrating against $v^\flat$ (the 1-form dual to $v$) we get
\[
\| \L_v g \|^2_{L^2} + \frac{1}{2} \| \diff^*(v^\flat) \|^2_{L^2} = 0
\] 
It follows that $h = \L_v g = 0$ and so $g$ is non-degenerate as required.
\end{proof}

\bibliographystyle{amsplain}
\bibliography{rigidity_4d_PE_bibliography}

\vspace{\baselineskip}

\noindent
\textsc{Joel Fine\\
Département de mathématique\\
Université libre de Bruxelles\\}
\href{mailto:joel.fine@ulb.be}{\tt{joel.fine@ulb.be}}

\end{document}